\documentclass[11pt,a4paper,reqno]{amsart}
\usepackage{amsmath,amsthm}

\newtheorem{theorem}{Theorem}[section]
\newtheorem{lemma}[theorem]{Lemma}

\newtheorem{corollary}[theorem]{Corollary}

\theoremstyle{remark}
\newtheorem{remark}[theorem]{\em \bf Remark}
\numberwithin{equation}{section}

\begin{document}

\title {Hyperstability of Jordan triple derivations on Banach algebras}

\author{Sang Og Kim}
\address{Department of Mathematics, Hallym University, Chuncheon 200-702,
Republic of Korea}
\email{sokim@hallym.ac.kr}

\author{ Abasalt Bodaghi}
\address{Department of Mathematics, Garmsar Branch, Islamic Azad University, Garmsar, Iran}
\email{abasalt.bodaghi@gmail.com}


\abstract{In this article, it is proved that a functional equation of (linear) Jordan triple derivations on unital Banach
algebras  under quite natural and simple assumptions is hyperstable. It is also shown that under some mild conditions approximate Jordan triple derivations on unital semiprime Banach algebras are (linear) derivations.}
\endabstract


\keywords{Hyers-Ulam stability, hyperstability, Jordan triple derivation, semiprime algebra  \\Corresponding author: A. Bodaghi}

\maketitle

\section{Introduction}
The stability problem of functional equations originated from a question of Ulam \cite{ulm} in 1940, concerning the stability of group homomorphisms: Let $G_1$ be a group and $G_2$ be a metric group with metric $d$. Given $\epsilon > 0$, does there exist a $\delta> 0$ such that if a mapping $\phi:G_1\longrightarrow G_2$ satisfies the inequality $d(\phi(st),\phi(s)\phi(t))<\delta$ for all $s,t\in G_1$, then there exists a homomorphism $\psi:G_1\longrightarrow G_2$ with $d(\phi(s),\psi(s))<\epsilon$ for all $s\in G_1$? In other words, under what conditions does there exist a homomorphism near an approximate homomorphism?  In 1941, Hyers \cite{hye} gave the first affirmative answer to the problem of Ulam for Banach spaces. The result states that
if $\delta>0$ and $f:X\longrightarrow Y$ is a mapping between Banach spaces $X,Y$ such that $\Vert f(x+y)-f(x)-f(y)\Vert\leq\delta$
for all $x,y\in X$, then there exists a unique additive mapping $T:X\longrightarrow Y$ such that $\Vert f(x)-T(x)\Vert\leq\delta$ for all $x\in X$.
Hyers' result was generalized by Aoki \cite{aok} for additive mappings and by Rassias \cite{ras} for linear mappings by considering an unbounded Cauchy difference. The paper of Rassias had a lot of influence on the development of what we call generalized Hyers-Ulam stability or Hyers-Ulam-Rassias stability of functional equations. Forti \cite{for1} and G\v{a}vruta \cite{gav} provided further generalizations of the Rassias' theorem. In a similar fashion, one can define many approximate  mappings. During the last decades several stability problems of functional equations were investigated. A large list of references concerning the  stability of various functional equations can be found e.g. in the books \cite{cze,hir,jun}.

Let $X$ be a real or complex algebra. An additive mapping $D:X\longrightarrow X$ is called a {\it derivation} if
$D(xy)=D(x)y+xD(y)$ for all $x,y\in X$. An additive mapping $D:X\longrightarrow X$ is said to be a {\it Jordan derivation} if $D(x^2)=D(x)x+xD(x)$ for all $x\in X$.  Also, an additive mapping $D:X\longrightarrow X$ is called a {\it Jordan triple derivation} if
$$ D(xyx)=D(x)yx+xD(y)x+xyD(x)\qquad(x,y\in X).$$
If moreover $D(\lambda x)=\lambda D(x)$ for all $\lambda$ and $x\in X$, then the definitions of {\it linear derivation, linear Jordan derivation} and {\it linear Jordan triple derivation} are self-explanatory.
It is easy to show that the following implications hold: derivation $\Rightarrow$ Jordan derivation $\Rightarrow$ Jordan triple derivation.
But the converse implications do not hold in general. The topic of approximate for various notions of derivations was studied by a number of mathematicians; see for example, \cite{bod},  \cite{boda}, \cite{bodp}, \cite{bad},  \cite{ene},  \cite{jnp},  \cite{pab},  \cite{ybk} and references therein.

Recall that in a situation where an approximate mapping must be a true mapping, we say that the functional equation of the mapping is {\it hyperstable}.

In this paper, we show that the functional equation of (linear) Jordan triple derivations on unital Banach algebras under what conditions is hyperstable. As some corollaries, we obtain that some approximate Jordan triple derivations on unital semiprime Banach algebras are (linear) derivations.

\section{Hyperstability of Jordan triple derivations I}

In this section we show that under some mild conditions the functional equation of Jordan triple derivations on unital Banach algebras can be hyperstable.
The following lemma is a special case of the general result presented by Forti \cite{for2}.

\begin{lemma}\label{lem1}{\rm(\cite{brz})} Let $S$ be a nonempty set, $(Y,d)$ a complete metric space and $\epsilon\geq 0$. Suppose that $h:S\longrightarrow Y, f:Y\longrightarrow Y, g:S\longrightarrow S$ and $\phi:S\longrightarrow [0,\infty)$ are mappings satisfying
\begin{eqnarray*}
d\left((f\circ h\circ g)(x),h(x)\right)\leq\phi(x), \\
d\Big(f(x),f(y)\Big)\leq\epsilon d(x,y),\\
\Phi(x):=\sum_{n=0}^{\infty}\epsilon^n\Big(\phi\circ g^n\Big)(x)<\infty,
\end{eqnarray*}
for all $x,y\in S$. Then, for every $x\in S$, the limit
\begin{eqnarray*}
H(x):=\lim_{n\to\infty}\Big(f^n\circ h\circ g^n\Big)(x)
\end{eqnarray*}
exists and $H:S\longrightarrow Y$, defined in this way, is the unique mapping such that $f\circ H\circ g=H$ and
\begin{eqnarray*}
d(h(x),H(x))\leq \Phi(x),\qquad (x\in S).
\end{eqnarray*}
\end{lemma}

Let $\mathbb{F}$ be the real or complex field and $\mathcal{B}_{\mathbb{F}}$ stand for the family of all sets $\Lambda\subseteq\mathbb{F}$ such that each additive function $f:\mathbb{F}\longrightarrow X$, which is bounded on $\Lambda$, must be continuous. It is well-known that if $\Lambda\subseteq\mathbb{F}$ and int$\Lambda\neq\emptyset$, then $\Lambda\in\mathcal{B}_{\mathbb{F}}$.

\begin{lemma}{\rm(\cite{brf})}\label{lem2} Let $X$ be a normed algebra over the complex field $\mathbb{C}$ and $Y$ a Banach algebra over $\mathbb{C}$.
Let $\Lambda_0\in\mathcal{B}_{\mathbb{C}}$ be a bounded set and $f:X\longrightarrow Y$ an additive function such that $f(\lambda x)=\lambda f(x)$ for all $x\in X$ and all $\lambda\in\Lambda_0$. Then, $f$ is $\mathbb{C}$-linear.
\end{lemma}


We mention that in the proof of the following theorem, we use some ideas of Theorem 1 of \cite{brf}.

\begin{theorem}\label{th1} Let $X$ be a Banach algebra with unit element $\bf {1}$. Let $h:X\longrightarrow X$ and $\phi,\psi:X^2\longrightarrow [0,\infty)$ satisfy
\begin{equation}\label{trd}
\Vert h(xyx)-h(x)yx-xh(y)x-xyh(x)\Vert\leq\psi(x,y),
\end{equation}
\begin{equation}\label{add}
\Vert h(\lambda ax+by)-\lambda Ah(x)-Bh(y)\Vert\leq\phi(x,y)
\end{equation}
for all $x,y\in X$, $\lambda\in\Lambda$ and for some fixed $a,b,A,B\in\mathbb{F}$, in which $ab\neq 0$, $\Lambda\subseteq\mathbb{F}, \Lambda\neq\emptyset$.
Assume that there is $\xi\in\Lambda\backslash\{0\}$ such that $d:=\xi A+B=\xi a+b\neq 0,1$ and
$$\Phi(x):=\sum_{k=0}^{\infty}\vert d\vert^{-k-1}\phi\Big(d^kx,d^kx\Big)<\infty,\label{4j}$$
\begin{align}
\liminf_{k\to\infty}|d|^{-k}\phi\Big(d^kx,d^ky\Big)=0,\label{nj}
\end{align}
\begin{align}
\liminf_{k\to\infty}|d|^{-3k}\psi\Big(d^k x,d^ky\Big)=0,\label{3j}
\end{align}
\begin{align}
\liminf_{k\to\infty}|d|^{-2k}\psi\Big(d^kx,y\Big)=0\label{2j}\qquad\forall x,y\in X.
\end{align}
Then, $h$ is a Jordan triple derivation. Furthermore, if $\Lambda$ has a bounded subset $\Lambda_0\in\mathcal{B}_{\mathbb{F}}$, $h$ is a linear Jordan triple derivation.
\end{theorem}
\begin{proof} Taking $x=y$ in \eqref{add}, we get
\begin{equation*}
\left\Vert h\Big((\lambda a+b)x\Big)-(\lambda A+B)h(x)\right\Vert\leq\phi(x,x)\label{x=y}
\end{equation*}
for all $x\in X, \lambda\in\Lambda$.
We have
\begin{align*}
 \left\Vert d^{-1}h(dx)-h(x)\right\Vert&\leq \vert d\vert^{-1}\Vert h(dx)-dh(x)\Vert\\
 & \leq \vert d\vert^{-1}\phi(x,x)
 \end{align*}
for all $x\in X$.  Let $f:X\longrightarrow X$ be given by $f(x)=\frac 1dx$, and $g:X\longrightarrow X$ be
defined by $g(x)=dx$ for all $x\in X$. Using Lemma \ref{lem1}, we deduce that the limit
$$ H(x):=\lim_{k\to\infty} \frac 1{d^k}h\Big(d^kx\Big)$$
exists and the mapping $H:X\longrightarrow X$ defined in this way is the unique solution of the functional equation
$ dH(x)=H(dx)$ such that $ \Vert h(x)-H(x)\Vert\leq \Phi(x)$ for all $x\in X$. We have
\begin{align*}
{}&\Vert H(xyx)-H(x)yx-xH(y)x-xyH(x)\Vert\\
&=\lim_{k\to\infty}\left\Vert \frac 1{d^{3k}}\Big[h\Big(d^{3k}xyx\Big)-d^{2k}h\Big(d^kx\Big)yx-d^{2k}xh\Big(d^ky\Big)x-d^{2k}xyh\Big(d^kx\Big)\Big]\right\Vert\\
&\leq\liminf_{k\to\infty}\frac 1{|d|^{3k}}\psi\Big(d^kx,d^ky\Big),
\end{align*}
for all $x,y\in X$. Hence, by \eqref{3j} we get
\begin{equation}\label{H}
H(xyx)=H(x)yx+xH(y)x+xyH(x)
\end{equation}
for all $x,y\in X$. We now show that $H$ is additive. By \eqref{add}, we obtain
\begin{align*}
{}&\left\Vert\frac 1{d^k}h\Big(d^k(\lambda ax+by)\Big)-\lambda A\frac 1{d^k}h\Big(d^kx\Big)-B\frac 1{d^k}h\Big(d^ky\Big)\right\Vert\leq\frac 1{|d|^k}\phi\Big(d^kx,d^ky\Big)
\end{align*}
for all $x,y\in X$. Since $\liminf_{k\to\infty}\frac 1{|d|^k}\phi\Big(d^kx,d^ky\Big)=0$, by \eqref{nj}
it follows that
\begin{equation}\label{eqadd}
H(\lambda ax+by)=\lambda AH(x)+BH(y) 
\end{equation}
for all $x,y\in X$ and $\lambda\in \Lambda$. Taking $\lambda=\xi$ in \eqref{eqadd}, we have  $H(0)=0$, $H(\xi ax)=\xi AH(x)$ and $H(by)=BH(y)$ for all $x,y\in X$.
So, we find
\begin{align*}
H(x+y)&=H\Big((\xi a)(\xi a)^{-1}x+bb^{-1}y\Big)\\
&=\xi AH\Big((\xi a)^{-1}x\Big)+BH\Big (b^{-1}y\Big)\\
&=H(x)+H(y),
\end{align*}
for all $x,y\in X$, and hence
\begin{equation}\label{additivity}
H(x+y)=H(x)+H(y)
\end{equation}
for all $x,y\in X$. Also, the inequality \eqref{trd} implies that
\begin{align*}{}& \left\Vert\frac 1{d^{2k}}\Big[h\Big(d^{2k}xyx\Big)-d^{k}h\Big(d^kx\Big)yx-d^{2k}xh(y)x-d^kxyh\Big(d^kx\Big)\Big]\right\Vert\leq\frac 1{|d|^{2k}}\psi\Big(d^kx,y\Big),
\end{align*}
for all $x,y\in X$. By \eqref{2j}, we arrive at
\begin{equation}\label{Hhh}
H(xyx)-H(x)yx-xh(y)x-xyH(x)=0
\end{equation}
for all $x,y\in X$. It follows from \eqref{H} and \eqref{additivity} that $H$ is a Jordan triple derivation.
Moreover, by \eqref{H} and \eqref{Hhh}, we have $ xH(y)x=xh(y)x$ for all $x,y\in X$. Taking $x=\bf{1}$, we obtain $H(y)=h(y)$ for all $y\in X$. Therefore, $h$ itself is a Jordan triple derivation. Finally, assume that $\Lambda$ has a bounded subset $\Lambda_0\in\mathcal{B}_{\mathbb{F}}$. The equality \eqref{eqadd} implies that $ H(\lambda ax)=\lambda AH(x)$ for all $x\in X$ and  $\lambda\in \Lambda$. Then, for each $x\in X$, the function $H_x:\mathbb{F}\longrightarrow X$, given by $H_x(\lambda)=H(\lambda x)$ is bounded on $\Lambda_0$. Since $H_x$ is additive, this means that $H_x$ is $\mathbb{R}$-linear. This gives the linearity of $H=h$ for $\mathbb{F}=\mathbb{R}$. In the case $\mathbb{F}=\mathbb{C}$, the linearity of $h$ can be obtained by using \eqref{eqadd} and Lemma \ref{lem2}.
\end{proof}

\begin{corollary} Suppose that $h: X\longrightarrow X$ is a mapping fulfilling 
\begin{align*}
{}&\Vert h(xyx)-h(x)yx-xh(y)x-xyh(x)\Vert\leq\theta(\Vert x\Vert^p+\Vert y\Vert^q+\Vert x\Vert^r\Vert y\Vert^s),\\
&\Vert h(x+y)-h(x)-h(y)\Vert\leq\theta(\Vert x\Vert^p+\Vert y\Vert^q+\Vert x\Vert^r\Vert y\Vert^s)
\end{align*}
for all $x,y\in X$ and for some  $0<\theta,p,q,r,s<1$ with $r+s<1$. Then, $h$ is a Jordan triple derivation.
\end{corollary}

\begin{proof} Putting $A=B=a=b=1,\xi=1$ and $\Lambda=\{1\}$ in Theorem \ref{th1}, we have $d=2$. It is easy to see that the conditions of Theorem
\ref{th1} are satisfied.
\end{proof}


Recall that an algebra $X$ is called {\it semiprime} whenever $aXa=\{0\}$ for some $a\in X$, we have $a=0$. The class of semiprime algebras contains $B(X)$ for Banach spaces $X$, the standard algebras, the subalgebras of $B(X)$ containing all finite rank operators on $X$ and all $C^*$-algebras. One remembers that a ring $R$ is said to be {\it $2$-torsion free} if and only if $2r=0$ implies $r=0$ for all $r\in R$.

\begin{corollary}\label{co1} Let $X$ be a semiprime Banach algebra with unit element $\bf {1}$. Let $h:X\longrightarrow X$ and $\phi,\psi:X^2\longrightarrow [0,\infty)$ satisfy
\begin{equation*}
\Vert h(xyx)-h(x)yx-xh(y)x-xyh(x)\Vert\leq\psi(x,y),
\end{equation*}
\begin{equation*}
\Vert h(\lambda ax+by)-\lambda Ah(x)-Bh(y)\Vert\leq\phi(x,y)
\end{equation*}
for all $x,y\in X$, $\lambda\in\Lambda$ and for some fixed $a,b,A,B\in\mathbb{F}, ab\neq 0$ and $\Lambda\subseteq\mathbb{F}$ with $\Lambda\neq\emptyset$.
Assume that there is $\xi\in\Lambda\backslash\{0\}$ such that $d:=\xi A+B=\xi a+b\neq 0,1$ and
\begin{align*}
{}&\Phi(x):=\sum_{k=0}^{\infty}\vert d\vert^{-k-1}\phi\Big(d^kx,d^kx\Big)<\infty,\\
&\liminf_{k\to\infty}|d|^{-k}\phi\Big(d^kx,d^ky\Big)=0,\\
&\liminf_{k\to\infty}|d|^{-3k}\psi\Big(d^k x,d^ky\Big)=0,\\
&\liminf_{k\to\infty}|d|^{-2k}\psi\Big(d^kx,y\Big)=0
\end{align*}
for all $x,y\in X$. Then, $h$ is a derivation. Furthermore, if $\Lambda$ has a bounded subset $\Lambda_0\in\mathcal{B}_{\mathbb{F}}$, then $h$ is a linear derivation.
\end{corollary}

\begin{proof}By Theorem \ref{th1}, $h$ is a Jordan triple derivation. Then, the conclusion follows directly by  \cite[Theorem 4.3]{bre}, which states that every  Jordan triple derivation on a $2$-torsion free semiprime ring is a derivation.
\end{proof}


One shoule remember that a function $\phi:A\longrightarrow B$ is called {\it contractively subadditive} if the domain and the codomain $(B,\leq)$ are closed under addition and there exists a constant $L$ with $0<L<1$ such that $ \phi(x+y)\leq L(\phi(x)+\phi(y))$ for all $x,y\in A$. Note that if $\phi$ is contractively subadditive, then $\frac 12\phi(2x)\leq L\phi(x)$ for all $x\in A$.  The next theorem is motivated by \cite{nap}, in which the authors proved the generalized Hyers-Ulam stability of homomorphisms on Banach algebras.

\begin{theorem}\label{th2} Let $X$ be a Banach algebra.  Let $\phi,\psi:X^2\longrightarrow [0,\infty)$  and $f:X\longrightarrow X$ with $f(0)=0$ satisfy
\begin{equation*}
\Vert f(xyx)-f(x)yx-xf(y)x-xyf(x)\Vert\leq\psi(x,y),
\end{equation*}
\begin{equation*}
\Vert f(2x+y)+f(x+2y)-f(3x)-f(3y)\Vert\leq\phi(x,y)
\end{equation*}
for all $x,y\in X$. Assume that there exists $0<L<1$ such that
\begin{equation*}
\frac 12\phi(2x,2y)\leq L\phi(x,y),
\end{equation*}
\begin{equation}\label{cond3}
\lim_{k\to\infty}\frac 1{8^k}\psi\Big(2^kx,2^ky\Big)=0
\end{equation}
for all $x,y\in X$. Then, there exists a unique Jordan triple derivation $H:X\longrightarrow X$ such that
\begin{equation*}
\Vert f(x)-H(x)\Vert\leq\frac1{2-2L}\Phi(x),
\end{equation*}
 for all $x\in X$, where $\Phi(x)=\phi\Big(\frac x2,0\Big)+\phi\Big(-\frac x2,0\Big)+\phi\Big(\frac x2,-\frac x2\Big)+\phi\Big(-\frac x3,\frac{2x}3\Big)$.
\end{theorem}

\begin{proof}
As in the proof of Theorem 2.2 of \cite{nap}, the limit
$$ H(x):=\lim_{k\to\infty} \frac 1{2^k}f\Big(2^kx\Big)$$
exists and the mapping $H:X\longrightarrow X$ defined in this way is the unique additive solution such that
$ 2H(x)=H(2x)$ for all $x\in X$ and 
$$ \Vert f(x)-H(x)\Vert\leq \frac 1{2-2L}\Phi(x)$$
 for all $x\in X$. Then, we have
\begin{align*}
{}&\Vert H(xyx)-H(x)yx-xH(y)x-xyH(x)\Vert\\
&=\lim_{k\to\infty}\left\Vert\frac1{8^k}\Big[f\Big(8^kxyx\Big)-4^kf\Big(2^kx\Big)yx-4^kxf\Big(2^ky\Big)x-
4^kxyf\Big(2^kx\Big)\Big]\right\Vert\\
&\leq\frac 1{8^k}\psi\left(2^kx,2^ky\right)
\end{align*}
By \eqref{cond3}, the right hand side of the above inequality tends to $0$ as  $k$ tend to infinity. So, $H$ is a Jordan triple derivation.
\end{proof}


We have the upcoming result which is analogous to Theorem \ref{th2} for the hyperstability of Jordan triple derivations on unital semiprime Banach algebras.

\begin{theorem}\label{thh} Let $X$ be a semiprime Banach algebra with unit element $\bf{1}$. Let $\phi,\psi:X^2\longrightarrow [0,\infty)$ and $f:X\longrightarrow X$ with $f(0)=0$ satisfy
\begin{equation*}\Vert f(xyx)-f(x)yx-xf(y)x-xyf(x)\Vert\leq\psi(x,y),
\end{equation*}
\begin{equation*}\Vert f(2x+y)+f(x+2y)-f(3x)-f(3y)\Vert\leq\phi(x,y)
\end{equation*}
for all $x,y\in X$. Assume that there exists $0<L<1$ such that
\begin{equation*}\frac 12\phi(2x,2y)\leq L\phi(x,y),
\end{equation*}
\begin{equation*}\lim_{k\to\infty}\frac 1{8^k}\psi\Big(2^kx,2^ky\Big)=0,
\end{equation*}
\begin{equation}\label{cond4}\lim_{k\to\infty}\frac 1{4^k}\psi\Big(2^kx,y\Big)=0
\end{equation}
for all $x,y\in X$. Then, $f$ itself is a derivation.
\end{theorem}

\begin{proof}By Theorem \ref{th2}, $H$ is a Jordan triple derivation, that is
 \begin{equation}H(xyx)=H(x)yx+xH(y)x+xyH(x)\qquad (x,y\in X).\label{2trd}
\end{equation}
Now, by \eqref{cond4} we get
\begin{align*}
{}&\Vert H(xyx)-H(x)yx-xf(y)x-xyH(x)\Vert\\
&=\lim_{k\to\infty}\left\Vert\frac1{4^k}\Big[f\Big(4^kxyx\Big)-2^kf\Big(2^kx\Big)yx-4^kxf(y)x-2^kxyf\Big(2^kx\Big)\Big]\right\Vert\\
&\leq\frac 1{4^k}\psi\Big(2^kx,y\Big)\to 0\text{~as~}k\to\infty.
\end{align*}
Therefore
\begin{equation}\label{equ2}
H(xyx)=H(x)yx+xf(y)x+xyH(x)
\end{equation}
for all $x,y\in X$. By \eqref{2trd} and \eqref{equ2}, we have $x(f(y)-H(y))x=0$ for all $x,y\in X$. If we take $x=\bf{1}$, it follows that $f=H$ itself is a Jordan triple derivation and thus it is a derivation (see the proof of Corollary \ref{co1}).
\end{proof}

For a positive integer $n_0 \in \mathbb{N}$, we put $\mathbb{T}_{n_0}^1 :=\{e^{i\theta} \ |\
0\leq\theta\leq\frac{2\pi}{n_0} \}.$ To  achieve some results of this section, we use from the following lemma which is proved in \cite[Lemma 1.1]{esb}. 

\begin{lemma}\label{le3} Let $n_0$ be a positive integer and let $X,Y$ be linear vector spaces on $\mathbb{C}$. Suppose that $f:X\longrightarrow Y$ is an additive mapping. Then, $f$ is $\mathbb{C}$-linear if and only if $f(\lambda x)=\lambda f(x)$ for all $x\in X$ and $\lambda\in\mathbb{T}_{n_0}^1$.
\end{lemma}

The proof of the next lemma is similar to that of \cite[Lemma 2.4]{nap} but we include it for the sake of completeness.

\begin{lemma}\label{le1} Let $X,Y$ be complex linear spaces. A mapping $f:X\longrightarrow Y$ satisfies
\begin{equation}\label{ll} f(2\mu x+\mu y)+f(\mu x+2\mu y)=\mu[f(3x)+f(3y)]
\end{equation}
for all $x,y\in X$ and all $\mu\in\mathbb{T}_{n_0}^1$, if and only if $f$ is $\mathbb{C}$-linear.
\end{lemma}
\begin{proof}Assume that $f$ satisfies (\ref{ll}) for all $x,y\in X$ and all $\mu\in\mathbb{T}_{n_0}^1$. It is easy to check that $f(0)=0$ and thus $f$ is additive by \cite[Lemma 2.1]{nap}. Letting $y=0$ in (\ref{ll}) and using the additivity of $f$, we obtain $f(\mu x)=\mu f(x)$ for all $x\in X$ and all $\mu\in\mathbb{T}_{n_0}^1$. Then, $f$ is linear by Lemma \ref{le3}. The converse is clear.
\end{proof}
Applying Lemma \ref{le1}, we improve the result of Theorem \ref{thh} as follows.

\begin{theorem}\label{th4} Let $X$ be a semiprime complex Banach algebra with identity. Let $\phi,\psi:X^2 \longrightarrow [0,\infty)$ and $f:X\longrightarrow X$ with $f(0)=0$ satisfy
\begin{equation*}\Vert f(xyx)-f(x)yx-xf(y)x-xyf(x)\Vert\leq\psi(x,y),
\end{equation*}
\begin{equation}\Vert f(2\mu x+\mu y)+f(\mu x+2\mu y)-\mu[f(3x)+f(3y)]\Vert\leq\phi(x,y)\label{con2102}
\end{equation}
for all $x,y\in X$ and all $\mu\in\mathbb{T}_{n_0}^1$.
Assume that there exists $0<L<1$ such that
\begin{equation}\label{3ee} \frac 12\phi(2x,2y)\leq L\phi(x,y),
\end{equation}
\begin{equation*}\lim_{k\to\infty}\frac{\psi(2^kx,2^ky)}{8^k}=0,
\end{equation*}
\begin{equation*}\lim_{k\to\infty}\frac 1{4^k}\psi(2^kx,y)=0
\end{equation*}
for all $x,y\in X$ and all $\mu\in\mathbb{T}_{n_0}^1$. Then, $f$ itself is a linear derivation.
\end{theorem}

\begin{proof} By Theorem \ref{thh}, $f$ is a derivation. The relation \eqref{3ee} necessities that 
$$\lim_{k\to\infty}\frac 1{2^k}\phi\Big(2^kx,2^ky\Big)=0$$
for all $x,y\in X$. Using the additivity of $f$ and \eqref{con2102}, we find
\begin{equation*}
 \Vert f(2\mu x+\mu y)+f(\mu x+2\mu y)-\mu[f(3x)+f(3y)]\Vert\leq\frac 1{2^k}\phi(2^kx,2^ky)
\end{equation*}
for all $x,y\in X$ and $\mu\in\mathbb{T}_{n_0}^1$. Hence, $f$ is $\mathbb{C}$-linear by Lemma \ref{le1}.
\end{proof}


\begin{corollary}  Let $X$ be a semiprime complex Banach algebra with identity and $p,q,\theta$
be  real numbers with $p<2,q<1$. Suppose that $f:X\longrightarrow X$ is a mapping such that
$$\Vert f(xyx)-f(x)yx-xf(y)x-xyf(x)\Vert\leq\theta(\Vert x\Vert^p+\Vert y\Vert^p),$$
\begin{equation*}\Vert f(2\mu x+\mu y)+f(\mu x+2\mu y)-\mu[f(3x)+f(3y)]\Vert\leq\theta(\Vert x\Vert^q+\Vert y\Vert^q)
\end{equation*}
for all $x,y\in X$ and all $\mu\in\mathbb{T}_{n_0}^1$.
Then, $f$ is a linear derivation.
\end{corollary}

\begin{proof} It is clear that $f(0)=0$, so it is enough to note that conditions of Theorem \ref{th4} hold with $L=2^{q-1}$.
\end{proof}
In parallel with the above corollary we have the next result.

\begin{corollary} Let $X$ be a semiprime complex Banach algebra with unit identity and $p,q,\theta$ be real numbers with $p<\frac 32, q<\frac 12$. Suppose that $f:X\longrightarrow X$ is a mapping such that
$$\Vert f(xyx)-f(x)yx-xf(y)x-xyf(x)\Vert\leq\theta\Vert x\Vert^p\Vert y\Vert^p,$$
\begin{equation*}\Vert f(2\mu x+\mu y)+f(\mu x+2\mu y)-\mu[f(3x)+f(3y)]\Vert\leq\theta\Vert x\Vert^q\Vert y\Vert^q
\end{equation*}
for all $x,y\in X$ and all $\mu\in\mathbb{T}_{n_0}^1$.
Then, $f$ is a linear derivation.
\end{corollary}

\begin{proof} Obviously, $f(0)=0$. The result follows from  Theorem \ref{th4} when $L=2^{2q-1}$.
\end{proof}

\begin{remark}  If the condition
$$\Vert f(2x+y)+f(x+2y)-f(3x)-f(3y)\Vert\leq\phi(x,y)$$
in Theorem \ref{th2} is replaced by
$$\Vert f(2\mu x+\mu y)+f(\mu x+2\mu y)-\mu[f(3x)+f(3y)]\Vert\leq\phi(x,y)$$
for all $x,y\in X$ and all $\mu\in\mathbb{T}_{n_0}^1$,
then the Jordan triple derivation $H$ is a linear Jordan triple derivation.
\end{remark}

\section{Hyperstability of Jordan triple derivations II}

In \cite{evd}, Eskandani et al. determined the general solution of the following mixed additive and quadratic functional equation
\begin{equation}\label{zam}
f(x+2y)+f(x-2y)+8f(y)=2f(x)+4f(2x).
\end{equation}
In \cite{bok2}, the authors established the Hyers-Ulam stability of the functional equation
\begin{equation}
f(x+my)+f(x-my)=2f(x)-2m^2f(y)+m^2f(2y)\label{3ae}
\end{equation}
in non-Archimedean normed spaces when $m$ is a nonzero even integer. Note that the equations \eqref{3ae} is a general form of \eqref{zam} (see also \cite{bok1}).

In this section we consider the hyperstability of Jordan triple derivations with the functional equation
\begin{equation*}
f(x+my)+f(x-my)=2f(x)-2m^2f(y)+m^2f(2y).\end{equation*}

Recall that a non-trivial ring (algebra) $R$ is called {\it prime} if for any two elements $a$ and $b$ of $R$, $arb=0$ for all $r\in R$ implies that either $a=0$ or $b=0$. Clearly, prime algebras are semiprime ones.


\begin{theorem}\label{th5} Let $X$ be a unital prime Banach algebra with a nontrivial idempotent element. Let $h:X\longrightarrow X$ be an odd mapping and $\phi,\psi:X^2\longrightarrow [0,\infty)$ satisfy
\begin{equation}\label{4trd}
\Vert h(xyx)-h(x)yx-xh(y)x-xyh(x)\Vert\leq\psi(x,y),
\end{equation}
\begin{equation}\label{4m}
\left\Vert h(x+my)+h(x-my)-2h(x)+2m^2h(y)-m^2h(2y)\right\Vert\leq\phi(x,y)
\end{equation}
for all $x,y\in X$ and some even integer constant $m\neq 0$.
Assume that
\begin{align}
{}&\Phi(x):=\sum_{k=0}^{\infty}2^{-k}\phi\Big(0,2^kx\Big)<\infty,\label{41}\\
&\liminf_{k\to\infty}2^{-3k}\psi\Big(2^k x,2^ky\Big)=0,\label{42}\\
&\liminf_{k\to\infty}2^{-2k}\psi\Big(2^kx,y\Big)=0\label{43}
\end{align}
for all $x,y\in X$. Then, $h$ is a  derivation.
\end{theorem}

\begin{proof}Putting $x=0$ in \eqref{4m}, we have
\begin{equation}\label{42mul}
\Vert 2h(y)-h(2y)\Vert\leq\frac 1{m^2}\phi(0,y)
\end{equation}
for all $y\in X$. Letting $y=2^kx$ in \eqref{42mul} and then dividing by $2^{k+1}$, we get
\begin{equation*}
\left\Vert \frac 1{2^k}h\Big(2^kx\Big)-\frac 1{2^{k+1}}h\Big(2^{k+1}x\Big)\right\Vert\leq\frac 1{2^{k+1}m^2}\phi\Big(0,2^kx\Big)
\end{equation*}
for all $x\in X$ and non-negative integer $k$. So, we obtain
\begin{equation*}
\left\Vert \frac 1{2^k}h\Big(2^kx\Big)-\frac 1{2^{k+j}}h\Big(2^{k+j}x\Big)\right\Vert\leq\frac 1{m^2}\left[\frac{\phi\Big(0,2^{k+1}x\Big)}{2^{k+1}}+\cdots+\frac{\phi\Big(0,2^{k+j}x\Big)}{2^{k+j}}\right].
\end{equation*}
This implies that $\left\{\frac {h(2^kx)}{2^k}\right\}$ is a Cauchy sequence in $X$ by \eqref{41}. Hence, there exists a mapping $H$ such that
\begin{equation*}
H(x):=\lim_{k\to\infty}\frac {h(2^kx)}{2^k}\qquad (x\in X).
\end{equation*}
On the other hand,
\begin{align*}
{}&\Vert H(xyx)-H(x)yx-xH(y)x-xyH(x)\Vert\\
&=\lim_{k\to\infty}\left\Vert\frac1{8^k}\Big[h\Big(8^kxyx\Big)-4^kh\Big(2^kx\Big)yx-4^kxh\Big(2^ky\Big)x-
4^kxyh\Big(2^kx\Big)\Big]\right\Vert\\
&\leq\liminf_{k\to\infty}\frac 1{8^k}\psi\Big(2^kx,2^ky\Big)=0.\\
\end{align*}
Thus
\begin{equation}\label{4xyx}
H(xyx)=H(x)yx+xH(y)x+xyH(x)
\end{equation}
for all $x,y\in X$. By \cite[Corollary 3.9]{jif}, which asserts that for a $2$-torsion free prime ring $R$ containing a nontrivial idempotent element, every mapping $G:R\longrightarrow R$ satisfying \eqref{4xyx} is additive. In particular, the mapping $H:X\longrightarrow X$ is additive. Hence, it follows that $H$ is a Jordan triple derivation. Now, by \eqref{43} we have
\begin{align*}
{}&\Vert H(xyx)-H(x)yx-xh(y)x-xyH(x)\Vert\\
&=\lim_{k\to\infty}\left\Vert\frac1{4^k}\Big[h\Big(4^kxyx\Big)-2^kh\Big(2^kx\Big)yx-4^kxh(y)x-2^kxyh\Big(2^kx\Big)\Big]\right\Vert\\
&\leq\liminf_{k\to\infty}\frac 1{4^k}\psi\Big(2^kx,y\Big)=0,
\end{align*}
which implies that
\begin{equation}\label{4xfx}
H(xyx)=H(x)yx+xh(y)x+xyH(x)
\end{equation}
for all $x,y\in X$. By \eqref{4xyx} and \eqref{4xfx}, we have $x\Big(h(y)-H(y)\Big)x=0$ for all $x,y\in X$. Putting $x=\bf{1}$, we see that $h=H$ itself is a Jordan triple derivation. Therefore, the conclusion follows directly by \cite[Theorem 4.3]{bre} (see the proof of Corollary \ref{co1}).
\end{proof}


\begin{theorem}\label{th6} Let $X$ be a unital prime Banach algebra over $\mathbb{C}$ with a nontrivial idempotent element. Let $h:X\longrightarrow X$ be an odd mapping and $\phi,\psi:X^2\longrightarrow [0,\infty)$ satisfy
\begin{equation*}
\Vert h(xyx)-h(x)yx-xh(y)x-xyh(x)\Vert\leq\psi(x,y),
\end{equation*}
\begin{equation}\label{5m}
\left\Vert h\Big(\mu(x+my)\Big)+h\Big(\mu(x-my)\Big)-\mu\Big(2h(x)-2m^2h(y)+m^2h(2y)\Big)\right\Vert\leq\phi(x,y)
\end{equation}
for all $x,y\in X, \mu\in\mathbb {T}_{n_0}^1$ and some even integer constant $m\neq 0$.
Assume that \eqref{41}, \eqref{42} and \eqref{43} hold with
\begin{equation}\label{53}
\liminf_{k\to\infty}2^{-k}\phi(2^kx,0)=0.
\end{equation}
for all $x,y\in X$. Then, $h$ is a  linear derivation.
\end{theorem}

\begin{proof} By Theorem \ref{th5}, $h$ is a  derivation. Letting $y=0$ in \eqref{5m}, we have
$$\Vert 2h(\mu x)-2\mu h(x)\Vert\leq\phi(x,0).$$
for all $x\in X$. Then, by applying the additivity of $h$ and \eqref{53}, it follows that $h(\mu x)=\mu h(x)$ for all $x\in X$ and $\mu \in\mathbb {T}_{n_0}^1$. Therefore,  by Lemma \ref{le3} $h$ is linear.
\end{proof}

The following corrollaries are the direct consequences of Theorem \ref{th5} and Theorem \ref{th6}, respectively. So, we omit their proof.

\begin{corollary} Let $X$ be a unital prime Banach algebra with a nontrivial idempotent element. Let $h:X\longrightarrow X$ be an odd mapping  such that
\begin{equation*}
\Vert h(xyx)-h(x)yx-xh(y)x-xyh(x)\Vert\leq\theta(\Vert x\Vert^p+\Vert y\Vert^q)
\end{equation*}
\begin{equation*}
\left\Vert h(x+my)+h(x-my)-2h(x)+2m^2h(y)-m^2h(2y)\right\Vert\leq\theta(\Vert x\Vert^p+\Vert y\Vert^q)
\end{equation*}
for all $x,y\in X$, for some constants $\theta, p, q, m$, where $m$ is nonzero even integer and $p<2, q<1$. Then, $h$ is a derivation.
\end{corollary}


\begin{corollary} Let $X$ be a unital prime Banach algebra over $\mathbb{C}$ with a nontrivial idempotent element. Let $h:X\longrightarrow X$ be an odd mapping  such that
\begin{equation*}
\Vert h(xyx)-h(x)yx-xh(y)x-xyh(x)\Vert\leq\theta(\Vert x\Vert^p+\Vert y\Vert^q)
\end{equation*}
\begin{equation*}
\left\Vert h\Big(\mu(x+my)\Big)+h\Big(\mu(x-my)\Big)-\mu\Big(2h(x)-2m^2h(y)+m^2h(2y)\Big)\right\Vert\leq\theta(\Vert x\Vert^p+\Vert y\Vert^q)
\end{equation*}
for all $x,y\in X, \mu\in\mathbb{T}_{n_0}^1$, for some constants $\theta, p, q, m$, where $m$ is nonzero even integer and $p<1, q<1$. Then, $h$ is a linear derivation.
\end{corollary}

\section*{Acknowledgment}
The authors express their sincere thanks to the reviewers for the careful and detailed reading of the manuscript and very
helpful suggestions that improved the manuscript substantially.

\providecommand{\bysame}{\leavevmode\hbox
to3em{\hrulefill}\thinspace}

\end{document}